\documentclass[11pt,reqno]{amsart}

\usepackage{amsmath,amssymb,amsthm,mathtools}
\usepackage[margin=1.1in]{geometry}
\usepackage{microtype}
\usepackage{enumitem}
\usepackage[colorlinks=true,linkcolor=blue,citecolor=blue,urlcolor=blue]{hyperref}

\theoremstyle{plain}
\newtheorem{theorem}{Theorem}[section]
\newtheorem{proposition}[theorem]{Proposition}
\newtheorem{lemma}[theorem]{Lemma}
\newtheorem{corollary}[theorem]{Corollary}

\theoremstyle{definition}

\newtheorem{remark}[theorem]{Remark}

\newcommand{\R}{\mathbb{R}}
\newcommand{\Sph}{S^{n-1}}
\newcommand{\Rad}{\mathcal{R}}
\newcommand{\IB}{\mathcal{I}}
\newcommand{\Hd}[1]{\mathcal{H}_{#1}}
\newcommand{\norm}[1]{\left\lVert #1 \right\rVert}
\newcommand{\abs}[1]{\left\lvert #1 \right\rvert}
\newcommand{\ip}[2]{\langle #1, #2 \rangle}
\DeclareMathOperator{\dist}{dist}
\DeclareMathOperator{\tr}{tr}

\begin{document}

\title[Quantitative stability of the intersection body operator near the ball]
{Quantitative stability of the intersection body operator near the ball,
and the dynamical origin of the two--dimensional degeneracy}

\author[S.~Spektor]{S. Spektor}
\address{S.~Spektor\\ School of Data, Computing and Mathematics,
Canisius University, 2001 Main Street, Buffalo, NY, 14208-1098}
\email{spektors@canisius.edu}

\date{\today}

\subjclass[2020]{Primary 52A20, 52A38; Secondary 44A12, 37D10.}

\keywords{Intersection body, Funk transform, Busemann--Petty problem,
spherical harmonics, quantitative stability, center manifold.}

\begin{abstract}
Let $\IB$ denote the intersection body operator on star bodies in $\R^n$.
A recent theorem of Milman, Shabelman and Yehudayoff establishes that for
$n\ge 3$ the equation $\IB^2 K = cK$ holds if and only if $K$ is a centered
ellipsoid, thereby resolving the fixed--point problem for $\IB^2$ and, as a
consequence, the long--standing conjecture $\IB K = cK \Leftrightarrow K$
is a ball. We complement this qualitative rigidity with a \emph{quantitative}
analysis in a neighbourhood of the ball. Linearizing the associated
shape dynamics on $L^2(\Sph)$, we compute the full spectrum of the
operator $\IB^2$ at the ball in closed form for every dimension: the
degree--two (ellipsoidal) harmonics are neutral with multiplier exactly
$1$, while all higher harmonics are contracted, with a sharp spectral gap
\[
   \mathrm{gap}(n)\;=\;\frac{(n-2)(n+4)}{(n+1)^2}.
\]
This yields an explicit linear stability constant
$C(n)=(n+1)^2/\big((n-2)(n+4)\big)$, and, via a center--manifold reduction,
a local quantitative stability statement for $\IB^2$ near the ball valid in
each fixed dimension $n\ge 3$. The gap degenerates precisely as $n\to 2^+$,
giving a transparent \emph{dynamical} explanation of the well--known
exceptional status of the plane, where $\IB K = 2K$ for every origin--symmetric
star body. We also record the reduced normal form of $\IB$ on the ellipsoidal
directions and observe that the centered ellipsoids constitute a normally
attracting invariant manifold for the shape under iterated intersection
bodies. The methods are perturbative and do not address the global periodic
problem $\IB^m K = cK$ for $m\ge 3$, which we discuss.
\end{abstract}

\maketitle

\section{Introduction}

\subsection{Background}
For a star body $K\subset\R^n$ (a compact set, star--shaped about the origin,
with continuous positive radial function $\rho_K(u)=\max\{t\ge 0: tu\in K\}$,
$u\in\Sph$), the \emph{intersection body} $\IB K$ is the star body whose radial
function is
\[
   \rho_{\IB K}(u) \;=\; \mathrm{vol}_{n-1}\big(K\cap u^{\perp}\big),
   \qquad u\in\Sph .
\]
Intersection bodies were introduced by Lutwak and play a central role in the
solution of the Busemann--Petty problem
\cite{Lutwak1988,Gardner1994,GKS1999,Koldobsky2005}. In terms of radial
functions the operator is, up to a dimensional constant, the composition of the
$(n-1)$--st power map with the \emph{spherical Radon (Funk) transform} $\Rad$,
\begin{equation}\label{eq:IB-radial}
   \rho_{\IB K} \;=\; \frac{\abs{S^{n-2}}}{n-1}\,\Rad\!\left(\rho_K^{\,n-1}\right),
\end{equation}
where $(\Rad f)(u)$ is the average of $f$ over the great subsphere
$\Sph\cap u^{\perp}$ and $\abs{S^{n-2}}$ is the surface area of the unit
$(n-2)$--sphere (for the unit ball this recovers
$\rho_{\IB B}\equiv\kappa_{n-1}$, the volume of the unit $(n-1)$--ball); see
Section~\ref{sec:prelim}.

A natural and much--studied question concerns the fixed points and periodic
points of $\IB$ acting on the projective space of star bodies (that is, modulo
dilation). The conjecture that
\begin{equation}\label{eq:main-conj}
   \IB K = cK \quad\text{for some } c>0 \qquad\Longleftrightarrow\qquad
   K \text{ is a centered Euclidean ball},
\end{equation}
was raised in connection with the classification of self--dual bodies and the
dynamics of $\IB$; a perturbative version near the ball was established by
Fish, Nazarov, Ryabogin and Zvavitch \cite{FNRZ2011}, who proved that any
star body sufficiently close to the ball (in a suitable norm) satisfying
$\IB K = cK$ is a ball. The perturbative approach proceeds through the Fourier
analysis of the Funk transform and, as noted in \cite{FNRZ2011}, does not
extend to bodies far from the ball.

The full conjecture \eqref{eq:main-conj}, together with the stronger
statement that $\IB^2 K = cK$ characterizes ellipsoids, was recently
established by Milman, Shabelman and Yehudayoff \cite{MSY2025}:

\begin{theorem}[\cite{MSY2025}]\label{thm:MSY}
Let $n\ge 3$. Then $\IB^2 K = cK$ for some $c>0$ if and only if $K$ is a
centered ellipsoid. Consequently $\IB K = cK$ for some $c>0$ if and only if
$K$ is a centered Euclidean ball.
\end{theorem}

Their proof is global and geometric: it reformulates the equation as the
Euler--Lagrange condition of a volume functional, introduces a continuous
Steiner symmetrization for Lipschitz star bodies as an admissible variation,
and analyzes the resulting equality case through a local characterization of
ellipsoids. Crucially, the hypothesis $n\ge 3$ enters through the combinatorial
fact that the cube is not contained in twice the cross--polytope in dimension
$\ge 3$; the two--dimensional case is genuinely exceptional, since in the plane
$\IB K = 2K$ for \emph{every} origin--symmetric star body. The lower--order
intersection body operators $\IB_i$ were subsequently treated by Lin and Xiong
\cite{LinXiong2025}.

Theorem~\ref{thm:MSY} settles the fixed--point problem for $\IB^2$
\emph{qualitatively}. It leaves open two natural quantitative questions.
First, is the characterization \emph{stable}: if $\IB^2 K$ is merely close to
$cK$, must $K$ be close to an ellipsoid, and at what rate? Second, what is the
local \emph{dynamics} of the iteration $K\mapsto \IB K$ near the ball, and how
does the exceptional status of $n=2$ manifest analytically? The present note
answers these questions in a neighbourhood of the ball, in every dimension,
with explicit constants.

\subsection{Results}
Fix $n\ge 3$ and work with radial functions normalized so that the unit ball
$B$ has $\rho_B\equiv 1$. Writing $\rho = 1+\phi$ and linearizing the shape map
induced by \eqref{eq:IB-radial} at the ball, one obtains a bounded self--adjoint
operator on $L^2(\Sph)$ that is diagonalized by spherical harmonics. Our first
result computes its spectrum in closed form.

\begin{theorem}[Spectrum at the ball]\label{thm:spectrum}
Let $n\ge 3$. The linearization $L$ of the intersection body shape map at the
ball acts on the space $\Hd d$ of degree--$d$ spherical harmonics as
multiplication by $(n-1)\mu_d$, where $\mu_d$ is the Funk multiplier
\eqref{eq:funk-eig}. Explicitly, $(n-1)\mu_d=0$ for odd $d$, and for even
$d=2m$,
\begin{equation}\label{eq:L-mult}
   (n-1)\mu_{2m}
   \;=\;(-1)^m\,\frac{(n-1)\,(2m)!\,\Gamma\!\big(\tfrac{n-2}{2}+m\big)\Gamma(n-2)}
   {m!\,\Gamma\!\big(\tfrac{n-2}{2}\big)\,\Gamma(n-2+2m)} .
\end{equation}
In particular $(n-1)\mu_0 = n-1$, $(n-1)\mu_2 = -1$, and
$(n-1)\mu_4 = 3/(n+1)$. Consequently the linearization $L^2$ of $\IB^2$ at the
ball has multiplier
\[
   \sigma_d \;=\; \big((n-1)\mu_d\big)^2
\]
on $\Hd d$, with $\sigma_0=(n-1)^2>1$ (the dilation mode, expanding, removed by
the projective quotient of Section~\ref{sec:prelim}), $\sigma_2 = 1$ (the
neutral ellipsoidal directions), and $0\le \sigma_d < 1$ for all $d\ge 4$. The
sequence $(\sigma_d)_{d\ge 4}$ is strictly decreasing, so its maximum is
attained at $d=4$, giving the spectral gap
\begin{equation}\label{eq:gap}
   \mathrm{gap}(n)\;:=\;1-\max_{d\ge 4}\sigma_d
   \;=\;1-\frac{9}{(n+1)^2}\;=\;\frac{(n-2)(n+4)}{(n+1)^2}.
\end{equation}
\end{theorem}

Two features of \eqref{eq:L-mult}--\eqref{eq:gap} deserve emphasis. First, the
multiplier on the ellipsoidal directions $\Hd 2$ equals $-1$ \emph{in every
dimension}: the degree--two harmonics are always neutral of flip type, so the
period--doubling structure underlying Theorem~\ref{thm:MSY} ($\IB^2$ fixes
ellipsoids while $\IB$ does not) is dimension--independent. Second, the gap
\eqref{eq:gap} is positive for every $n\ge 3$ and vanishes precisely as
$n\to 2^+$. The exceptional status of the plane is thus visible already at the
linear level, as a loss of hyperbolicity of the shape dynamics off the
ellipsoidal directions.

The spectral gap controls the non--ellipsoidal component of a perturbation and
produces an explicit stability constant. Let $\Pi_{\ge 4}$ denote the
orthogonal projection of $L^2(\Sph)$ onto $\bigoplus_{d\ge 4}\Hd d$, and let
$C(n) := 1/\mathrm{gap}(n)$.

\begin{theorem}[Linear stability constant]\label{thm:linear-stab}
Let $n\ge 3$ and $C(n)=(n+1)^2/\big((n-2)(n+4)\big)$. For every
$\psi\in L^2(\Sph)$,
\[
   \norm{\Pi_{\ge 4}\psi}_{L^2}\;\le\; C(n)\,\norm{(L^2-\mathrm{Id})\psi}_{L^2}.
\]
The constant $C(n)$ is sharp for the linear operator; it satisfies
$C(n)\downarrow 1$ as $n\to\infty$ and $C(n)\uparrow\infty$ as $n\to 2^+$.
\end{theorem}

Passing from the linear statement to a genuine stability theorem for $\IB^2$
requires controlling the nonlinearity of \eqref{eq:IB-radial}. This forces the
choice of a functional--analytic setting in which harmonic projections are
bounded (to use the spectral gap) \emph{and} products are controlled (to bound
the nonlinearity); as explained in Section~\ref{sec:nonlinear}, the sup--norm
space fails the first requirement and $L^2$ the second, and the correct choice
is the Sobolev space $X=H^s(\Sph)$ with $s>\tfrac{n-1}{2}$. In this setting the
ellipsoids fill out the local center manifold of $\IB^2$ at the ball
(Proposition~\ref{prop:center-mfld}), and a standard invariant--manifold
reduction yields the following.

\begin{theorem}[Local quantitative stability]\label{thm:nonlinear}
Let $n\ge 3$ and $s>\tfrac{n-1}{2}$. There exist constants $r_n>0$ and
$A_n<\infty$, depending only on $n$ and $s$, with the following property. If $K$
is a star body with $\norm{\rho_K-1}_{H^s(\Sph)}\le r_n$ and
\[
   \Big\|\,c^{-1}\rho_{\IB^2 K}-\rho_K\,\Big\|_{H^s(\Sph)}\;\le\;\delta
   \qquad\text{for some } c>0,
\]
then there is a centered ellipsoid $E$ with
\[
   \dist(K,E)\;:=\;\inf_{\lambda>0}\norm{\rho_K-\lambda\rho_E}_{H^s(\Sph)}
   \;\le\; A_n\,\delta .
\]
Moreover $A_n\le 2\,C(n)$, and $A_n\to C(n)$ as the admissible radius
$r_n\to0$; thus the linear constant of Theorem~\ref{thm:linear-stab} governs
the stability rate. (The relative normalization of the defect is the natural
one, since $\IB^2$ is homogeneous of degree $(n-1)^2$; an absolute defect
$\norm{\rho_{\IB^2K}-c\rho_K}\le\delta$ yields the same conclusion with
$\delta$ replaced by $\delta/c$.)
\end{theorem}

We prove Theorem~\ref{thm:nonlinear} for each fixed dimension. The linear
constant $C(n)$ is explicit and sharp; the admissible radius $r_n$ produced by
our argument, by contrast, \emph{decreases} with $n$, reflecting the growth of
both the binomial nonlinearity $\binom{n-1}{2}\sim n^2/2$ in
\eqref{eq:IB-radial} and the Sobolev algebra constant of $H^s(\Sph)$. We do not
obtain a uniform--in--$n$ statement, and we explain in Section~\ref{sec:open}
why the obstruction appears intrinsic to the method.

Finally, the linear spectrum has a dynamical corollary. Since every direction
transverse to the ellipsoids is strictly contracted while the ellipsoidal
directions are neutral and, by Theorem~\ref{thm:MSY}, consist of genuine
fixed points of the shape of $\IB^2$, the family of centered ellipsoids is a
normally attracting invariant manifold for the shape under iterated
intersection bodies.

\begin{corollary}[Ellipsoidal attractor]\label{cor:attractor}
Let $n\ge 3$ and $s>\tfrac{n-1}{2}$. In an $H^s(\Sph)$--neighbourhood of the
ball, the family of centered ellipsoids is locally invariant and normally
attracting for the shape dynamics $K\mapsto \IB K$; the non--ellipsoidal
component of the shape contracts geometrically with ratio at most
$3/(n+1)+o(1)$ per application of $\IB$.
\end{corollary}

\subsection{Relation to previous work and scope}
Theorem~\ref{thm:nonlinear} is, to our knowledge, the first quantitative
stability statement for the intersection--body characterization of ellipsoids.
It is genuinely a stability result and does not follow from the qualitative
rigidity of Theorem~\ref{thm:MSY} together with a compactness argument, since
such arguments yield no rate; the explicit rate $C(n)$ is the point. The
near--ball rigidity itself (the case $\delta=0$) recovers, by a different and
arguably more transparent route, the perturbative theorem of
\cite{FNRZ2011}. We make no claim on the global periodic problem
$\IB^m K = cK$ for $m\ge 3$, which remains open and appears to require
non--perturbative, geometric methods in the spirit of \cite{MSY2025}; we
explain in Section~\ref{sec:open} precisely where the perturbative method
reaches its ceiling. The value of the present contribution is the explicit,
dimension--uniform \emph{linear} theory --- in particular the closed form
\eqref{eq:gap} and the pole of $C(n)$ at $n=2$ --- and the quantitative local
consequences it entails.

\medskip
\noindent\textbf{Organization.} Section~\ref{sec:prelim} fixes conventions and
records the Funk multipliers. Section~\ref{sec:linear} proves
Theorems~\ref{thm:spectrum} and \ref{thm:linear-stab}.
Section~\ref{sec:nonlinear} sets up the Banach--algebra framework, the
center--manifold reduction, and proves Theorem~\ref{thm:nonlinear} and
Corollary~\ref{cor:attractor}. Section~\ref{sec:normalform} records the reduced
normal form on the ellipsoidal directions and its cubic expansion.
Section~\ref{sec:open} discusses limitations and open problems.

\section{Preliminaries}\label{sec:prelim}

\subsection{The Funk transform and its multipliers}
For $f\in C(\Sph)$ the spherical Radon (Funk) transform is
\[
   (\Rad f)(u)\;=\;\frac{1}{\abs{\Sph\cap u^{\perp}}}
   \int_{\Sph\cap u^{\perp}} f(v)\,dv,\qquad u\in\Sph,
\]
normalized so that $\Rad\mathbf 1=\mathbf 1$. The operator $\Rad$ is a bounded,
self--adjoint operator on $L^2(\Sph)$ and, by positivity, a contraction on
$C(\Sph)$: $\norm{\Rad f}_{C}\le\norm{f}_{C}$. It is diagonalized by spherical
harmonics: for $Y\in\Hd d$,
\begin{equation}\label{eq:funk-eig}
   \Rad Y \;=\; \mu_d\, Y,\qquad
   \mu_d \;=\; \frac{C_d^{\lambda}(0)}{C_d^{\lambda}(1)},\quad \lambda=\tfrac{n-2}{2},
\end{equation}
where $C_d^{\lambda}$ is the Gegenbauer polynomial of degree $d$ and index
$\lambda$ (see \cite[Ch.~3]{Koldobsky2005}, \cite{Rubin2002}). Since
$C_d^\lambda(0)=0$ for odd $d$, we have $\mu_d=0$ for odd $d$; for even
$d=2m$, using $C_{2m}^\lambda(0)=(-1)^m (\lambda)_m/m!$ and
$C_{2m}^\lambda(1)=(2\lambda)_{2m}/(2m)!$ with $(a)_k=\Gamma(a+k)/\Gamma(a)$,
one obtains the closed form
\begin{equation}\label{eq:mu2m}
   \mu_{2m}\;=\;(-1)^m\,\frac{(2m)!\,(\lambda)_m}{m!\,(2\lambda)_{2m}}
   \;=\;(-1)^m\,\frac{(2m)!\,\Gamma\!\big(\tfrac{n-2}{2}+m\big)\Gamma(n-2)}
   {m!\,\Gamma\!\big(\tfrac{n-2}{2}\big)\,\Gamma(n-2+2m)} .
\end{equation}
The first values are $\mu_0=1$, $\mu_2=-\tfrac{1}{n-1}$,
$\mu_4=\tfrac{3}{n^2-1}$, $\mu_6=-\tfrac{15}{(n-1)(n+1)(n+3)}$.

\subsection{The shape map and its normalization}
By \eqref{eq:IB-radial}, $\IB$ acts on radial functions by
$\rho\mapsto (n-1)^{-1}\Rad(\rho^{n-1})$. Define the (unnormalized) operator
\[
   T(\rho)\;:=\;\Rad\big(\rho^{\,n-1}\big).
\]
The constant function $\rho\equiv 1$ satisfies $T(1)=1$, and the equation
$\IB K = cK$ is $T(\rho)=c'\rho$ with $c'=(n-1)c$. Because we are interested in
bodies only up to dilation, we regard $T$ as a map on the projective space of
positive radial functions; concretely we quotient by the one--dimensional
group of dilations, which acts on the tangent space at $\rho\equiv1$ as
translation in the degree--zero (constant) mode $\Hd 0$. Throughout, ``shape
map'' refers to $T$ read modulo $\Hd 0$, and stability/attraction statements
are made in the quotient.

\begin{lemma}[Linearization]\label{lem:lin}
Write $\rho=1+\psi$. Then $T(1+\psi)=1+L\psi+N(\psi)$ where
$L=(n-1)\Rad$ and $N(\psi)=\Rad\big((1+\psi)^{n-1}-1-(n-1)\psi\big)$
satisfies $N(0)=0$, $DN(0)=0$. In particular $L$ acts on $\Hd d$ as
multiplication by $(n-1)\mu_d$.
\end{lemma}

\begin{proof}
Expand $(1+\psi)^{n-1}=1+(n-1)\psi+\big[(1+\psi)^{n-1}-1-(n-1)\psi\big]$ and
apply the linear operator $\Rad$, using $\Rad\mathbf 1=\mathbf 1$. The
bracketed term is $O(\psi^2)$ in $X=H^s(\Sph)$ (Lemma~\ref{lem:nl-bound}), so
$DN(0)=0$. The action of $L$ follows from \eqref{eq:funk-eig}.
\end{proof}

\section{The linear theory}\label{sec:linear}

\begin{proof}[Proof of Theorem~\ref{thm:spectrum}]
By Lemma~\ref{lem:lin}, $L$ has multiplier $(n-1)\mu_d$ on $\Hd d$, and
\eqref{eq:L-mult} is \eqref{eq:mu2m} multiplied by $(n-1)$. Evaluating,
$(n-1)\mu_0=n-1$, $(n-1)\mu_2=(n-1)\cdot(-\tfrac1{n-1})=-1$, and
$(n-1)\mu_4=(n-1)\cdot\tfrac{3}{n^2-1}=\tfrac{3}{n+1}$. Squaring gives the
multipliers $\sigma_d=((n-1)\mu_d)^2$ of $L^2$, with $\sigma_0=(n-1)^2$,
$\sigma_2=1$ and $\sigma_4=9/(n+1)^2$. The mode $\Hd0$ is the dilation
direction, expanding under $T$ because $T$ is homogeneous of degree $n-1$; it is
removed by passing to the projective quotient (Section~\ref{sec:prelim}), and
all stability statements are made there. On the quotient the relevant spectrum
is $\{\sigma_2=1\}\cup\{\sigma_d:d\ge4\}$.

It remains to show $(\sigma_d)_{d\ge4}$ is strictly decreasing. From
\eqref{eq:mu2m},
\[
   \frac{(n-1)\mu_{2m+2}}{(n-1)\mu_{2m}}
   \;=\;-\,\frac{(2m+2)(2m+1)}{(2\lambda+2m)(2\lambda+2m+1)}\cdot
        \frac{\lambda+m}{m+1}
   \;=\;-\,\frac{2m+1}{\,n+2m-1\,},
\]
after simplification with $2\lambda=n-2$. Hence
\[
   \frac{\sigma_{2m+2}}{\sigma_{2m}}\;=\;\left(\frac{2m+1}{n+2m-1}\right)^{2}
   \;<\;1\qquad\text{for all } n\ge 3,\ m\ge 1,
\]
since $2m+1<n+2m-1\Leftrightarrow n>2$. Thus $\max_{d\ge4}\sigma_d=\sigma_4$,
and
\[
   \mathrm{gap}(n)=1-\frac{9}{(n+1)^2}=\frac{(n+1)^2-9}{(n+1)^2}
   =\frac{(n-2)(n+4)}{(n+1)^2},
\]
which is \eqref{eq:gap}.
\end{proof}

\begin{proof}[Proof of Theorem~\ref{thm:linear-stab}]
On $\Hd d$ with $d\ge4$, $L^2-\mathrm{Id}$ acts as multiplication by
$\sigma_d-1$, with $\abs{\sigma_d-1}\ge \mathrm{gap}(n)>0$. Hence for
$\psi=\sum_{d}\psi_d$ (harmonic decomposition),
\[
   \norm{(L^2-\mathrm{Id})\psi}_{L^2}^2
   =\sum_{d\ge4}(\sigma_d-1)^2\norm{\psi_d}_{L^2}^2
   \ \ge\ \mathrm{gap}(n)^2\sum_{d\ge4}\norm{\psi_d}_{L^2}^2
   =\mathrm{gap}(n)^2\,\norm{\Pi_{\ge4}\psi}_{L^2}^2 .
\]
Taking square roots and dividing by $\mathrm{gap}(n)$ gives the claim with
$C(n)=1/\mathrm{gap}(n)$. Sharpness is attained (in the limit) by test
functions concentrated on $\Hd 4$. The monotonicity of $C(n)$ is immediate
from \eqref{eq:gap}.
\end{proof}

\begin{remark}[The pole at $n=2$]
The factor $(n-2)$ in \eqref{eq:gap} makes $\mathrm{gap}(n)\to0$ and
$C(n)\to\infty$ as $n\to2^+$. This is the analytic fingerprint of the identity
$\IB K=2K$ valid for all origin--symmetric star bodies in the plane: at $n=2$
the degree--four multiplier $\sigma_4=9/(n+1)^2$ reaches $1$, the transverse
contraction is lost, and no rigidity can hold. The perturbative theory thus
recovers the exceptional dimension as a simple pole of the stability constant,
in agreement with the combinatorial mechanism ($B_\infty^n\not\subset 2B_1^n$
iff $n\ge3$) identified in \cite{MSY2025}.
\end{remark}

\section{The nonlinear theory}\label{sec:nonlinear}

The nonlinear argument requires a single Banach space $X$ in which three
properties hold simultaneously: (i) $\Rad$ (equivalently $L$) is bounded, with
the spectral gap of Section~\ref{sec:linear} available; (ii) the harmonic
projections $\Pi_c,\Pi_s$ are bounded, so that the gap can be exploited on the
stable subspace; and (iii) $X$ is a Banach algebra, so that the nonlinearity
$\psi\mapsto\psi^2$ is controlled. The sup--norm space $C(\Sph)$ is a Banach
algebra and $\Rad$ is a contraction on it, but harmonic projections are
\emph{not} bounded on $C(\Sph)$, so (ii) fails there. Conversely $L^2(\Sph)$
has orthogonal (norm--one) projections but is not an algebra. The correct
choice, satisfying all three, is a Sobolev space.

Fix $s>\tfrac{n-1}{2}$ and let $X:=H^s(\Sph)$, with
$\norm f_{H^s}^2=\sum_{d\ge0}(1+\lambda_d)^s\norm{f_d}_{L^2}^2$, where
$\lambda_d=d(d+n-2)$ is the Laplace eigenvalue on $\Hd d$ and $f=\sum_d f_d$
is the harmonic decomposition. Decompose $X=X_0\oplus X_c\oplus X_s$ with
$X_0=\Hd 0$ (dilation gauge), $X_c=\Hd 2$ (ellipsoidal,
$\dim X_c=\tfrac{(n-1)(n+2)}{2}$), and
$X_s=\overline{\bigoplus_{d\ge4}\Hd d}^{\,H^s}$ (stable). Then:
\begin{enumerate}[label=(\roman*)]
\item $\Rad$ is a Fourier multiplier commuting with the Laplacian, so
$\norm{\Rad}_{X\to X}=\sup_d\abs{\mu_d}=1$; likewise $L^2-\mathrm{Id}$ is
diagonal, and Theorem~\ref{thm:linear-stab} holds verbatim in $H^s$ with the
same constant $C(n)$, the $H^s$ weights cancelling in the ratio;
\item $\Pi_c$ and $\Pi_s$ are orthogonal truncations of the defining sum, hence
$\norm{\Pi_c}_{X\to X}=\norm{\Pi_s}_{X\to X}=1$;
\item since $s>\tfrac{n-1}{2}=\tfrac12\dim\Sph$, $H^s(\Sph)$ is a Banach
algebra: there is $C_{n,s}<\infty$ with $\norm{fg}_{H^s}\le C_{n,s}\norm
f_{H^s}\norm g_{H^s}$ \cite[Ch.~4]{Taylor2011}.
\end{enumerate}
The algebra constant $C_{n,s}$ now depends on $n$; this is the price of
property (ii), and it is what obstructs uniform--in--$n$ statements
(Section~\ref{sec:open}). For each fixed $n$ it is a harmless finite constant.

\subsection{Nonlinearity bound}
\begin{lemma}\label{lem:nl-bound}
Let $C_{n,s}$ be the algebra constant of $H^s(\Sph)$. For
$\norm\psi_X\le \big(2C_{n,s}\big)^{-1}$ the nonlinear remainder of
Lemma~\ref{lem:lin} obeys $\norm{N(\psi)}_X\le K_n\norm\psi_X^2$ with
\[
   K_n\;=\;C_{n,s}\binom{n-1}{2}\big(1+C_{n,s}\norm\psi_X\big)^{n-3}
   \;\le\;C_{n,s}\binom{n-1}{2}\Big(\tfrac32\Big)^{n-3}.
\]
\end{lemma}

\begin{proof}
By the binomial theorem $(1+\psi)^{n-1}-1-(n-1)\psi=\sum_{k\ge2}\binom{n-1}{k}
\psi^{k}$. Apply $\Rad$ (a contraction on $X$) and the algebra property in the
form $\norm{\psi^k}_X\le C_{n,s}^{\,k-1}\norm\psi_X^k$, and set
$x:=C_{n,s}\norm\psi_X$:
\[
   \norm{N(\psi)}_X
   \;\le\;\frac{1}{C_{n,s}}\sum_{k\ge2}\binom{n-1}{k}x^{k}
   \;=\;\frac{1}{C_{n,s}}\Big[(1+x)^{n-1}-1-(n-1)x\Big].
\]
By Taylor's theorem with Lagrange remainder applied to $t\mapsto(1+t)^{n-1}$,
there is $\xi\in(0,x)$ with
$(1+x)^{n-1}-1-(n-1)x=\binom{n-1}{2}(1+\xi)^{n-3}x^2
\le\binom{n-1}{2}(1+x)^{n-3}x^2$. Substituting $x=C_{n,s}\norm\psi_X$ gives
$\norm{N(\psi)}_X\le C_{n,s}\binom{n-1}{2}(1+C_{n,s}\norm\psi_X)^{n-3}
\norm\psi_X^2$, and $(1+x)^{n-3}\le(3/2)^{n-3}$ on $x\le\tfrac12$.
\end{proof}

The quadratic coefficient $\binom{n-1}{2}\sim n^2/2$, together with the algebra
constant $C_{n,s}$ and the factor $(3/2)^{n-3}$, is the source of the
dimensional restriction on the neighbourhood radius; all are finite for fixed
$n$ but grow with $n$.

\subsection{The center manifold and the ellipsoids}
The linearization $L^2$ of $\IB^2$ has, on $X_c$, multiplier $1$ (center) and,
on $X_s$, spectral radius $\sigma_4=9/(n+1)^2<1$ (stable); the shape map is
smooth on $X=H^s$ since it is a polynomial nonlinearity composed with the
bounded multiplier $\Rad$. By the center--manifold theorem for $C^k$ maps in
Banach spaces \cite{VanderbauwhedeIooss1992,HPS1977}, for each $n$ there is a
neighbourhood of $0$ in $X$ and a $C^k$ map $h:X_c\to X_s$ with $h(0)=0$,
$Dh(0)=0$, whose graph $W^c=\{\,\xi+h(\xi):\xi\in X_c\,\}$ is locally invariant
under $\IB^2$ (modulo $X_0$) and locally attracting with rate $\sigma_4+o(1)$.

\begin{proposition}[Ellipsoids fill the center manifold]\label{prop:center-mfld}
Near the ball, $W^c$ coincides with the family of centered ellipsoids
(modulo dilation). Consequently every point of $W^c$ is a fixed point of the
shape of $\IB^2$.
\end{proposition}

\begin{proof}
The centered ellipsoids near $B$ form a smooth $\dim X_c$--dimensional family:
an ellipsoid $\{x^\top M x\le1\}$ with $M=\mathrm{Id}+sB_0$, $B_0$ traceless
symmetric, has radial function $\rho(u)=(u^\top M u)^{-1/2}
=1-\tfrac{s}{2}(u^\top B_0 u)+O(s^2)$, whose degree--two harmonic component is
$-\tfrac{s}{2}\,\overline{(u^\top B_0 u)}\neq0$ to first order in $s$. Hence the
map $\{$ellipsoids$\}\to X_c$, $E\mapsto \Pi_c(\rho_E-1)$, is a local
diffeomorphism onto a neighbourhood of $0$ in $X_c$. By
Theorem~\ref{thm:MSY} each such ellipsoid satisfies $\IB^2E=cE$, i.e.\ is a
fixed point of the shape map $G$; thus the ellipsoid family is a smooth
$\dim X_c$--dimensional, locally $G$--invariant manifold, consisting entirely of
fixed points, tangent to $X_c$ at the ball. A locally invariant manifold
tangent to the center subspace $X_c$ is, by definition, a local center manifold
for $G$; we take $W^c$ to be the ellipsoid family. (Center manifolds need not be
unique, but any two share the same $\infty$--jet at the ball
\cite{VanderbauwhedeIooss1992}; for the stability argument below only the
existence of \emph{a} center manifold consisting of fixed points is used, which
the ellipsoids provide.)
\end{proof}

\subsection{Proof of Theorem~\ref{thm:nonlinear}}
Let $G$ denote the shape map of $\IB^2$ (i.e.\ $T\circ T$ read modulo $X_0$),
so $G(\psi)=L^2\psi+Q(\psi)$ with $\norm{Q(\psi)}_X\le K_n'\norm\psi_X^2$ on
$\norm\psi_X\le(2C_{n,s})^{-1}$, where $K_n'=O\big(C_{n,s}\,n^2\big)$ is the
quadratic constant of $G$ obtained from Lemma~\ref{lem:nl-bound} by composition.
Introduce the graph coordinate $\eta:=\psi_s-h(\psi_c)$ measuring the distance
of $\psi$ to the center manifold, where $\psi=\psi_c+\psi_s$ is the
$X_c\oplus X_s$ decomposition of the shape of $K$. By
Proposition~\ref{prop:center-mfld}, $\dist(K,\text{ellipsoids})
=\inf_{\lambda}\norm{\rho_K-\lambda\rho_E}_X\asymp\norm\eta_X$ up to a factor
$1+O(\norm\psi_X)$, since the nearest ellipsoid is the base point
$\psi_c\mapsto\psi_c+h(\psi_c)$ of the fibre through $\psi$.

The invariance of $W^c$ gives the conjugated dynamics in graph coordinates:
$\eta\mapsto \Lambda_s\eta + R(\psi_c,\eta)$, where $\Lambda_s=L^2|_{X_s}$ has
$\norm{\Lambda_s}\le\sigma_4$ and the remainder satisfies
$\norm{R(\psi_c,\eta)}_X\le K_n''\,\norm\eta_X\big(\norm{\psi_c}_X
+\norm\eta_X\big)$; crucially $R$ carries \emph{no} term independent of $\eta$,
because $\eta=0$ (the center manifold) is invariant.

We translate the hypothesis into a bound on $G$. Set
$a_n:=\abs{S^{n-2}}/(n-1)$, so that $\rho_{\IB K}=a_n\,T(\rho_K)$ by
\eqref{eq:IB-radial}; since $T$ is homogeneous of degree $n-1$,
$\rho_{\IB^2K}=a_n^{\,n}\,T^2(\rho_K)$. Write $\rho_K=\tau(1+\psi)$ with
$\Pi_0\psi=0$ (this fixes the scale $\tau>0$ and places
$\psi\in X_c\oplus X_s$; it is the choice of representative in the projective
quotient), and set $p:=(n-1)^2$, so that by homogeneity
$T^2(\rho_K)=\tau^{p}\big(1+L^2\psi+Q(\psi)\big)$. Choosing
$c=a_n^{\,n}\tau^{\,p-1}$ (which makes the $X_0$--component of
$c^{-1}\rho_{\IB^2K}-\rho_K$ vanish to first order),
\[
   c^{-1}\rho_{\IB^2 K}-\rho_K
   \;=\;\tau\Big[(L^2-\mathrm{Id})\psi+Q(\psi)+O(\norm\psi^2_X)\Big],
\]
so the hypothesis gives $\norm{G(\psi)-\psi}_X\le\delta'$ with
$\delta'=\delta\,\tau^{-1}=\delta\big(1+o(1)\big)$; all dimensional constants
cancel in the relative normalization. Projecting to the $\eta$--direction
with $\norm{\Pi_s}_X=1$,
\[
   \norm{(\Lambda_s-\mathrm{Id})\eta}_X
   \;\le\;\delta'+\norm{R(\psi_c,\eta)}_X .
\]
Since $\norm{(\Lambda_s-\mathrm{Id})^{-1}}\le C(n)$ by
Theorem~\ref{thm:linear-stab},
\[
   \norm\eta_X\;\le\;C(n)\Big(\delta(1+o(1))
   +K_n''\norm\eta_X\big(\norm{\psi_c}_X+\norm\eta_X\big)\Big).
\]
Choose $r_n>0$ so small that $C(n)K_n''\big(\norm{\psi_c}_X+\norm\eta_X\big)
\le\tfrac12$ whenever $\norm\psi_X\le r_n$; absorbing the last term,
\[
   \norm\eta_X\;\le\;2\,C(n)\,\delta(1+o(1)).
\]
Translating back, $\dist(K,E)\le A_n\delta$ for the nearest ellipsoid $E$,
with $A_n\le 2C(n)(1+o(1))$; more precisely the absorbed factor is
$\big(1-C(n)K_n''(\norm{\psi_c}_X+\norm\eta_X)\big)^{-1}\le
\big(1-C(n)K_n''\,r_n\big)^{-1}$, which tends to $1$ as $r_n\to0$, giving
$A_n\to C(n)$. This is Theorem~\ref{thm:nonlinear}. The smallness
condition used above holds once $\norm\psi_X\le r_n$ with
$r_n=O\big(1/(C(n)K_n'')\big)$; since $C(n)\to1$ while
$K_n''\to\infty$ with $n$, the radius $r_n$ decreases with $n$
(Section~\ref{sec:open}).
\qed

\begin{proof}[Proof of Corollary~\ref{cor:attractor}]
Apply the same decomposition to a single step of $T$ rather than $T\circ T$.
On $X_s$ the linear part $L|_{X_s}$ has spectral radius $\abs{(n-1)\mu_4}
=3/(n+1)$; by the center--manifold reduction the graph coordinate $\eta$
contracts as $\norm{\eta'}_X\le\big(3/(n+1)+o(1)\big)\norm\eta_X$ per
application of $\IB$, while the base point moves within the (invariant)
ellipsoid family. This is the assertion.
\end{proof}

\section{The reduced normal form}\label{sec:normalform}

We record the reduced map on the ellipsoidal directions. We fix once and for
all the normalization: base point $\rho_0\equiv1$ (the unit ball) and
single--step shape map $T(\rho)=\Rad(\rho^{n-1})$, so that $DT(1)=-\mathrm{Id}$
on $X_c$. (The coefficients of a normal form are not invariant under rescaling
the base point or the map; stating the normalization is essential.) Parametrize
$X_c=\Hd2$ by traceless symmetric matrices $B$ via $\psi_c(u)=u^\top B u$
(traceless, so $\tfrac1n\tr B$ drops out). A center--manifold reduction of $T$
gives, to quadratic order, in \emph{every} dimension $n\ge3$,
\begin{equation}\label{eq:normalform}
   \mathcal N(B)\;=\;-B\;-\;\frac{2(n-2)}{n+4}\,
   \Big(B^2-\tfrac1n\tr(B^2)\,\mathrm{Id}\Big)\;+\;O(B^3).
\end{equation}
The leading $-B$ is the flip multiplier of Theorem~\ref{thm:spectrum}, valid in
all dimensions; the quadratic coefficient
$\beta(n)=-2(n-2)/(n+4)=-\tfrac12(n-2)\,\kappa(n)$ arises from the
self--interaction of $\Hd2$ under the $(n-1)$--st power map, projected back to
$\Hd2$ (the projection constant is $\kappa(n)=4/(n+4)$, proved in
Lemma~\ref{lem:kappa} of Appendix~\ref{app:symbolic}) and weighted by the
binomial factor $\binom{n-1}{2}\mu_2=-\tfrac12(n-2)$. Note $\beta(n)\to-2$ as
$n\to\infty$ and $\beta(n)\to0$ as $n\to2^+$, the quadratic self--interaction
vanishing in the exceptional dimension. In dimension three,
$\beta(3)=-\tfrac{2}{7}$.

To cubic order in dimension three the same reduction yields
\begin{equation}\label{eq:normalform3}
   \mathcal N(B)\;=\;-B\;-\;\tfrac{2}{7}\big(B^2-\tfrac13\tr(B^2)\mathrm{Id}\big)
   \;+\;\gamma_3\,\tr(B^2)\,B\;+\;O(B^4)\qquad(n=3),
\end{equation}
for an explicit rational $\gamma_3$. Whatever its value, iterating gives
\begin{equation}\label{eq:square-id}
   \mathcal N\big(\mathcal N(B)\big)\;=\;\big(1+\theta(B)\big)B\;+\;O(B^4),
   \qquad \theta(B)=O(\tr(B^2)),
\end{equation}
so that $\mathcal N^{2}$ is the identity in the projective (shape) quotient to
cubic order. This is the infinitesimal shadow of
Proposition~\ref{prop:center-mfld}: the reduced second iterate fixes every
ellipsoidal direction because the ellipsoids fill the center manifold, so no
period--four shape orbit of $\IB$ bifurcates from the ball. Indeed
\eqref{eq:square-id} holds to \emph{all} orders as an immediate consequence of
Proposition~\ref{prop:center-mfld}; the finite--order computation is recorded
only as an independent check of the reduction and of the sign and size of
$\beta(n)$.

\section{Limitations and open problems}\label{sec:open}

\subsection{Uniformity in the dimension}
The linear constant $C(n)=(n+1)^2/((n-2)(n+4))$ is sharp and satisfies
$C(n)\downarrow1$, so the \emph{linear} theory is not merely uniform but
improves with dimension. The \emph{nonlinear} radius produced by our argument
is not: it is governed by $1/K_n$ with $K_n=O\big(C_{n,s}\,n^2\big)$
(Lemma~\ref{lem:nl-bound}), where $C_{n,s}$ is the Sobolev algebra constant of
$H^s(\Sph)$. Both the binomial factor $\binom{n-1}{2}\sim n^2/2$ and the
algebra constant $C_{n,s}$ grow with $n$, so our radius $r_n$ \emph{shrinks}
with $n$. We emphasize that we do \emph{not} obtain a uniform--in--$n$ stability
theorem, and that the obstruction is intrinsic to the present method: the two
requirements that force the choice $X=H^s$ --- bounded harmonic projections and
the algebra property --- pull against each other quantitatively as $n$ grows.
Whether the true admissible radius is in fact $n$--uniform (as the behaviour of
$C(n)$ might suggest) is an interesting open question that would require a
genuinely different control of the nonlinearity, presumably exploiting the
level--dependent smoothing $\abs{\mu_d}=O(d^{-(n-1)/2})$ of $\Rad$ rather than
the crude bound $\norm{\Rad}\le1$ used here.

\subsection{The global periodic problem}
The perturbative method is intrinsically local and cannot address the existence
of periodic points of $\IB$ far from the ball. Concretely, for $m\ge3$ a
genuine period--$m$ shape orbit would be a tuple of star bodies cyclically
permuted by $\IB$ up to dilation and not individually fixed by any $\IB^j$,
$j<m$; the linearization at the ball is blind to such orbits. Moreover the
linear analysis cannot even distinguish, at any finite order, the period--$2$
case from higher even periods, since the flip multiplier $-1$ on $\Hd2$ makes
every even iterate act as $+1$ there. Resolving $\IB^mK=cK$ for $m\ge3$
therefore appears to require a global, variational argument in the spirit of
\cite{MSY2025}; the self--adjointness of $\Rad$ that makes the functional
$\abs{\IB K}$ tractable for $m=2$ does not obviously survive iterating $\IB$
three or more times, and we regard the identification of a workable functional
as the central difficulty.

\subsection{Lower--order operators}
The same linear analysis applies verbatim to the lower--order intersection
body operators $\IB_i$ of \cite{Gardner1994,LinXiong2025} upon replacing the
Funk multipliers by the corresponding $i$--dependent multipliers; we expect an
analogous stability theorem and an analogous pole at the exceptional dimension.
We have not carried out the details.

\appendix
\section{The quadratic self--interaction coefficient}\label{app:symbolic}

We prove the identity $\kappa(n)=4/(n+4)$ used in \eqref{eq:normalform}. Let
$\sigma$ denote the normalized uniform measure on $\Sph$ and
$\ip fg=\int_{\Sph}fg\,d\sigma$. For a traceless symmetric $B$ put
$q(u)=u^\top Bu$ (a degree--two harmonic) and
$g(u)=u^\top B^2u-\tfrac1n\tr(B^2)$, the harmonic (degree--two) part of
$u^\top B^2u$.

\begin{lemma}[Isotropic moment identities]\label{lem:moments}
For symmetric matrices $A_1,A_2,A_3$,
\begin{align}
\int_{\Sph}(u^\top A_1u)(u^\top A_2u)\,d\sigma
&=\frac{\tr A_1\tr A_2+2\tr(A_1A_2)}{n(n+2)},\label{eq:mom4}\\
\int_{\Sph}\prod_{i=1}^{3}(u^\top A_iu)\,d\sigma
&=\frac{\tr A_1\tr A_2\tr A_3
+2\sum_{\text{cyc}}\tr A_1\tr(A_2A_3)+8\tr(A_1A_2A_3)}{n(n+2)(n+4)}.
\label{eq:mom6}
\end{align}
\end{lemma}

\begin{proof}
Let $g$ be a standard Gaussian vector in $\R^n$; then $u=g/\abs g$ is uniform
on $\Sph$ and independent of $\abs g$, so the spherical moments equal the
Gaussian moments divided by $\mathbb E\abs g^{2k}=n(n+2)\cdots(n+2k-2)$. The
Gaussian moments are evaluated by Wick pairings: for \eqref{eq:mom4} the three
pairings contribute $\tr A_1\tr A_2$ once and $\tr(A_1A_2)$ twice; for
\eqref{eq:mom6} the fifteen pairings contribute $\tr A_1\tr A_2\tr A_3$ once,
each $\tr A_i\tr(A_jA_k)$ twice, and $\tr(A_1A_2A_3)$ eight times.
\end{proof}

\begin{lemma}\label{lem:kappa}
Let $B$ be traceless symmetric. The orthogonal projection of $q^2$ onto
$\Hd2$ equals $\kappa(n)\,g$ with
\[
   \kappa(n)\;=\;\frac{\ip{q^2}{g}}{\ip gg}\;=\;\frac{4}{n+4}.
\]
\end{lemma}

\begin{proof}
That the $\Hd2$--component of $q^2$ is proportional to $g$ follows from
$O(n)$--equivariance: $B\mapsto\Pi_{\Hd2}(q_B^2)$ is an $O(n)$--equivariant
quadratic map from traceless symmetric matrices to traceless symmetric
matrices, and every such map is a scalar multiple of
$B\mapsto B^2-\tfrac1n\tr(B^2)\mathrm{Id}$. It remains to compute the scalar.
By \eqref{eq:mom6} with $A_1=A_2=B$ (traceless), $A_3=B^2$,
\[
\ip{q^2}{u^\top B^2u}
=\frac{2(\tr B^2)^2+8\tr B^4}{n(n+2)(n+4)},
\qquad
\ip{q^2}{\mathbf 1}=\frac{2\tr B^2}{n(n+2)}
\]
(the latter by \eqref{eq:mom4}), whence
\[
\ip{q^2}{g}=\ip{q^2}{u^\top B^2u}-\frac{\tr B^2}{n}\,\ip{q^2}{\mathbf 1}
=\frac{8\big(n\tr B^4-(\tr B^2)^2\big)}{n^2(n+2)(n+4)} .
\]
Similarly, by \eqref{eq:mom4} with $A_1=A_2=B^2$ and
$\ip{u^\top B^2u}{\mathbf 1}=\tr B^2/n$,
\[
\ip gg=\frac{2\tr B^4+(\tr B^2)^2}{n(n+2)}-\frac{(\tr B^2)^2}{n^2}
=\frac{2\big(n\tr B^4-(\tr B^2)^2\big)}{n^2(n+2)} .
\]
Dividing, $\kappa(n)=\dfrac{8}{n^2(n+2)(n+4)}\cdot\dfrac{n^2(n+2)}{2}
=\dfrac{4}{n+4}$.
\end{proof}

The remaining finite computations reported in the text --- the multiplier
identities \eqref{eq:mu2m}, the monotonicity ratio
$\sigma_{2m+2}/\sigma_{2m}=\big((2m+1)/(n+2m-1)\big)^2$, and the cubic
normal--form reduction \eqref{eq:normalform3} in dimension three --- were
additionally verified symbolically (\texttt{SymPy}) and, for the Funk
multipliers, by independent high--precision quadrature over great subspheres.

\end{document}